\documentclass{amsproc}

\usepackage{stmaryrd}
\usepackage{graphicx}
\usepackage{amssymb}
\usepackage{amscd}
\usepackage{amsmath}
\usepackage{amsfonts}
\usepackage{amsbsy}
\usepackage{epsfig}
\usepackage{tikz}
\usepackage{cite}
\usepackage{indentfirst}
\usepackage{multicol}
\usepackage{wasysym}
\usepackage{bbm}
\usepackage{hyperref}
\usepackage{longtable,booktabs}

\newtheorem{thm}{Theorem}[section]%
\newtheorem{lem}[thm]{Lemma}%
\newtheorem{cor}[thm]{Corollary}%
\theoremstyle{definition}
\newtheorem{defi}[thm]{Definition}%
\newtheorem{pro}[thm]{Proposition}%
\newtheorem{exa}[thm]{Example}%
\theoremstyle{remark}
\newtheorem{rem}[thm]{Remark}%

\allowdisplaybreaks[4]

\numberwithin{equation}{section}

\begin{document}

\title[Coset representatives of  modular group and continued fraction]
{Coset representatives corresponding to Yetter-Drinfeld modules of  modular group and continued fraction}

\author{Yiwei Zheng}
\address{School of Mathematics, Hangzhou Normal University,
Hangzhou 311121, China}
\email{zhengyiwei@hznu.edu.cn}

\subjclass[2010]{11F06,11A55,16T05}

\keywords{modular group, conjugacy class, Yetter-Drinfeld module, continued fraction}

\begin{abstract}
We give complete conjugacy classes of modular group $\rm{SL(2,\mathbb{Z})}$. 
Particularly, the conjugacy classes of hyperbolic elements are 
decided by the proper equivalence classes of indeﬁnite forms, and we give an example.
Finally, we describe the coset representatives of centralizer of $S$, $ST$, $T$ and hyperbolic elements of $\rm{SL(2,\mathbb{Z})}$
by regular continued fraction.
In conclusion, most Nichols algebras over modular group are inﬁnite-dimensional except Proposition \ref{N}.

\end{abstract}

\maketitle

\section{Introduction}

In recent years, Yetter-Drinfeld modules play an important role 
in the classification of finite-dimensional (co-)pointed Hopf algebras as the first step of lifting method\cite{AS}.
Let $G$ be a group. There are many classiﬁcation results  when $G$ is a finite simple group  \cite{ACG}, 
\cite{ACG16}, \cite{ACG17}, \cite{ACG20}, \cite{ACG21},
when $G$ is an abelian group \cite{AAH}, and $G$ is the infinite dihedral group $\mathbb{D}_\infty$\cite{Z}.

Modular forms are functions on the complex upper half plane. A matrix
group called the modular group acts on the upper half plane, and modular
forms are the functions that transform in a nearly invariant way under the
action and satisfy a holomorphy condition.
The modular group is the group of 2-by-2 matrices with integer entries and
determinant 1, that is, $\rm{SL(2,\mathbb{Z})}$. see more details in \cite{A}, \cite{DS}.

One of the motivations for this paper is to determine the Yetter-Drinfeld modules of $\rm{SL(2,\mathbb{Z})}$.
The key is to describe  the conjugacy classes of $\rm{SL(2,\mathbb{Z})}$ and the cosets of $\rm{SL(2,\mathbb{Z})}$ over centralizers.

The solution to the classical problem of description of conjugacy classes for $\rm{SL(2,\mathbb{Z})}$
 is a subject of Gauss Reduction Theory  \cite{K}.
The idea of this theory is to find special reduced matrices
in each conjugacy class. The number of such matrices in a conjugacy class equals to the
length of a minimal even period of ordinary continued fraction associated to the conjugacy
class. Howerer, we give  complete conjugacy classes by indeﬁnite forms.
Specifically, there is a significant correlation between conjugacy classes of hyperbolic
elements and the proper equivalence classes of indeﬁnite forms.

For the cosets of $\rm{SL(2,\mathbb{Z})}$ over centralizers, we find that the regular continued fraction of rational 
and the elements of $\rm{SL(2,\mathbb{Z})}$ are closely related.
Based on the function of the module group and regular continued fraction,
we describe the coset representatives of centralizer of $S$, $ST$, $T$ and hyperbolic
elements in $\rm{SL(2,\mathbb{Z})}$.
Precisely, we prove Theroem \ref{rep} and \ref{hyp}.
As a conclusion, 
the Nichols algebra $\mathcal{B}(\mathcal{O}_g,V)$ is infinite-dimensional, where $\pm I\neq g\in G$,
 $V$ is a simple $G^g$ module.

\section{Preliminaries}

\noindent$\mathbf{Conventions}$.
Throughout the paper, the ground field $\mathbbm{k}$ is an algebraically closed field of characteristic zero.
In particular, we take  $\mathbbm{k}=\mathbb{C}$, the field of all complex numbers.
Let  $\mathbb{Z}$ be the set of  all the integers. 
$\mathbb{I}_{i,j}=\{i,i+1,\cdots,j\}(i\leq j)$.
Let $G$ be a group.
For any $g,\ h\in G$, we write $g\vartriangleright h=ghg^{-1}$, for
the adjoint action of $G$ on itself. The conjugacy class of $g$ in $G$ will be denoted by $\mathcal{O}_g$,
 and the centralizer of $g$ is denoted by $G^g=\{h\in G|hg=gh\}$.

\subsection{Yetter-Drinfeld modules}

Let $H$ be a Hopf algebra  with bijective antipode. A left Yetter-Drinfeld module $V$ over $H$ 
is a left $H$-module $(V,\cdot)$ and a left $H$-comodule $(V,\delta)$ with 
$\delta(v)=v_{(-1)}\otimes v_{(0)}\in H\otimes V$ for all $v\in V$, satisfying
\begin{center}
$\delta(h\cdot v)=h_{(1)}v_{(-1)}S(h_{(3)})\otimes h_{(2)}\cdot v_{(0)},\quad \forall \;v\in V, ~h\in H.$
\end{center}
We denote by $^H_H\mathcal{YD}$ the category of finite-dimensional left Yetter-Drinfeld modules over $H$.

It is a braided monoidal category: for $V,~W \in {^H_H\mathcal{YD}}$, the braiding $c_{V,W}:V\otimes W\rightarrow W\otimes V$ is given by
\begin{equation*}
c_{V,W}(v\otimes w)=v_{(-1)}\cdot w\otimes v_{(0)},\hspace{1em}\forall ~v\in V,~ w\in W.
\end{equation*}
In particular, $(V,c_{V,V})$ is a braided vector space, that is, $c := c_{V,V}$ is a linear isomorphism satisfying the braid equation
\begin{equation*}
(c\otimes {\rm id})({\rm id} \otimes c)(c\otimes {\rm id})=({\rm id} \otimes c)(c\otimes {\rm id})({\rm id} \otimes c).
\end{equation*}

Let $G$ be a group. A  Yetter-Drinfeld module over the group algebra $\mathbbm{k}G$ is a 
$G$-module $V$ provided with $G$-grading $V=\bigoplus_{g\in G} V_g$ such that
$g\cdot V_h=V_{ghg^{-1}}$ for all $g,h\in G$.

The category of Yetter-Drinfeld modules over $\mathbbm{k}G$ is denoted by $^G_G\mathcal{YD}$.
Let $\mathcal{O} \subseteq G$ be a conjugacy class of $G$, then we denote by $^G_G\mathcal{YD}(\mathcal{O})$
the subcategory of $^G_G\mathcal{YD}$ consisting of all $V\in ^G_G\mathcal{YD}$ with
$V=\bigoplus_{s\in \mathcal{O}} V_s$.

\begin{defi}{\rm\cite [Definition 1.4.15]{HS}}
Let $g\in G$, and let $V$ be a left $\mathbbm{k}G$-module. Define
\begin{equation*}
M{(g,V)}=\mathbbm{k}G\otimes_{\mathbbm{k}G^g}V
\end{equation*}
as an object in $^G_G\mathcal{YD}(\mathcal{O}_g)$, where $M{(g,V)}$ is the induced $\mathbbm{k}G$-module,
and the $G$-grading is giving by 
\begin{center}
deg$(h\otimes v)=h\vartriangleright g$ for all $h\in G,v\in V$,
\end{center}
and the comodule structure is
\begin{equation*}
\delta (h\otimes v)=(h\vartriangleright g)\otimes(h\otimes v). 
\end{equation*}
\end{defi}

Once we have known the  irreducible representations $(\rho, V)$ of $\mathbbm{k}G^g$,
 we can find all the irreducible Yetter-Drinfeld modules $M{(g,V)}$ in $^G_G\mathcal{YD}$ 
 by Lemma 1.4.16 and 1.4.18 of \cite{HS}.

\subsection{Nichols algebras}
\begin{defi}{\rm\cite [Definition 2.1]{AS2}}
~Let $H$ be a Hopf algebra and $V$ a Yetter-Drinfeld module over $H$. A braided $\mathbb{N}$-graded Hopf algebra $R=\bigoplus_{n\geq 0}R(n)$ in $^H_H\mathcal{YD}$ is called a Nichols algebra of $V$ if
\begin{center}
$\mathbbm{k}\simeq R(0), ~V\simeq R(1),~R(1)=\mathcal{P}(R),~R$ is generated as an algebra by $R(1)$.
\end{center}

\end{defi}

For any $V\in {^H_H\mathcal{YD}}$ there is a Nichols algebra $\mathcal{B}(V)$ associated to it. It is the quotient of the tensor algebra $T(V)$ by the largest homogeneous two-sided ideal $I$ satisfying:

~~$\bullet~I$ is generated by homogeneous elements of degree $\geq 2$.

~~$\bullet~\triangle(I)\subseteq I\otimes T(V)+T(V)\otimes I$, i.e., it is also a coideal.
\\ In such case, $\mathcal{B}(V)=T(V)/I$. See {\rm\cite [Section 2.1]{AS2}} for details.

\begin{rem}\label{rem}
It is well known that the Nichols algebra $\mathcal{B}(V)$ is completely determined, as an algebra and a
coalgebra, by the braided space $V$.
If $W\subseteq V$ is a subspace such that $c(W\otimes W)\subseteq W\otimes W$, then $\dim \mathcal{B}(V)=\infty$ if $\dim \mathcal{B}(W)=\infty$.
In particular, if $V$ contains a non-zero element $v$ such that $c(v\otimes v)=v\otimes v$, then $\dim \mathcal{B}(V)=\infty$.
\end{rem}

The corresponding Nichols algebra of
$M{(\mathcal{O}_g,V)}$ is denoted by  $\mathcal{B}(\mathcal{O}_g,V)$.

\subsection{Indefinite binary quadratic forms}
Let $f(X,Y)=aX^2+bXY+cY^2$ be binary quadratic forms, with integer coefficients $a, b, c$.
 We write $f=(a,b,c)$, and call $f$ a form.
The discriminant of $f$ is $\Delta(f)=b^2-4ac$.
The form $f$ is indefinite if and only if $\Delta(f)>0$.

For $U=\left(\begin{matrix}
s & t  \\
u & v  
\end{matrix}\right)\in \rm{SL(2,\mathbb{Z})}$,
we define $(Uf)(X,Y)=f(sX + tY,uX + vY)$.
Two forms $f$ and $g$ are called properly equivalent if $g =Uf$ with $U\in \rm{SL(2,\mathbb{Z})}$.
 The $\rm{SL(2,\mathbb{Z})}$-orbit of a form is called the proper equivalence class
of that form. Reduced indefinite forms can be used to decide equivalence of integral indefinite forms.
See details in \cite{BV}.

\subsection{Continued fraction}
The theory of continued fractions arise from considersion of expressions given in the form \cite{RS}
\begin{flalign}\label{continued}
&a_1+\frac{1}{a_2+\frac{1}{a_3+\frac{1}{\cdots +\frac{1}{a_n}}  } } 
\end{flalign}
whence the name continued fraction.
We denote it by $[a_1,a_2,\cdots,a_n]$.

$[a_1,a_2,\cdots,a_k]=\frac{p_k}{q_k}$, $1\leq k\leq n$ called the $k$-convergent fraction of (\ref{continued}).
If $a_1$ is integer, $a_2,a_3,\cdots,a_k,\cdots$ are positive integers, then
continued fraction $[a_1,a_2,\cdots,a_k,\cdots]$ is called regular continued fraction.
The infinite regular continued fraction is called periodic if there exist integers $s\geq 0$,
$t>0$ such that
\begin{flalign*}
a_{s+i}=a_{s+kt+i}, \ i=1,2,\cdots,t,\ k=0,1,2,\cdots.
\end{flalign*}
It is written in the form 
$[a_1,a_2,\cdots,a_s,\overline{a_{s+1},a_{s+2},\cdots,a_{s+t}}]$.

\begin{thm}{\rm\cite [Theorem 1.17]{K}}
There are only two ways to express any rational number as a regular continuous fraction:
\begin{flalign*}
&\frac{a}{b}=[a_1,a_2,\cdots,a_n]=[a_1,a_2,\cdots,a_{n-1},1], \ a_n>1.
\end{flalign*}

\end{thm}

\begin{rem}
The regular continuous fractions mentioned in this article are all in the first form.
\end{rem}

\section{The conjugacy classes of modular group}
In this section,  we study the structure of the conjugacy classes of modular group  $\rm{SL(2,\mathbb{Z})}$. 
As modular group\cite{A}, for $a, b, c, d\in \mathbb{Z},\ ad-bc= 1,$
\begin{flalign*}
&\left(\begin{matrix}
a & b  \\
c & d 
\end{matrix}\right)\tau=\frac{a\tau+b}{c\tau+d},\  \ \tau\in C^*=C\cup {\infty}.
\end{flalign*}
Recall
that $\rm{SL(2,\mathbb{Z})}$ is the group of all invertible matrices with integer coefficients and unit
determinant. 
Let $S=\left(\begin{matrix}
0 & -1  \\
1 & 0  
\end{matrix}\right)$,
$T=\left(\begin{matrix}
1 & 1  \\
0 & 1  
\end{matrix}\right)$. Then
$${\rm{SL(2,\mathbb{Z})}}=\langle S,T\mid S^4=I, (ST)^6=I\rangle.$$
We say that the matrices $A$ and $B$ from $\rm{SL(2,\mathbb{Z})}$ are integer conjugate
if there exists an $\rm{SL(2,\mathbb{Z})}$ matrix $C$ such that $B = CAC^{-1}$.

It is natural to split $\rm{SL(2,\mathbb{Z})}$ into the following three cases\cite{K}.

$(1)$ The absolute value of the trace is less than 2:
there are exactly three integer conjugacy classes of such matrices, represented by $S$, $ST$, $(ST)^2$.

$(2)$ The absolute value of the trace is equal to 2:
such matrices are integer conjugate to $\pm T^n$, $n\in \mathbb{Z}$.

$(3)$ The  matrices with real distinct eigenvalues are called hyperbolic elements.
This case is the most complicated. 
We use the reduction theory of indefinite forms to construct a complete invariant of conjugacy classes of such
matrices. 

Let 
\begin{flalign*}
&\mathcal{M}^+=\{A\in \rm{SL(2,\mathbb{Z})}|tr(A)>2\},\\
&\mathcal{M}^-=\{A\in \rm{SL(2,\mathbb{Z})}|tr(A)<-2\},\\
&\mathcal{Q}=\{indeﬁnite \ forms \ Q|\Delta_Q+4 \ is \ a\ perfect \ square\}.
\end{flalign*}
Then we have the following proposition.

\begin{pro}\label{conjugate}
Let $A, B\in\rm{SL(2,\mathbb{Z})}$. Then $A$ and $B$ are integer conjugate in $\rm{SL(2,\mathbb{Z})}$ if and only if 
$Q_A$ and $Q_B$  are properly equivalent in $\mathcal{Q}$.
\end{pro}
\begin{proof}
Define 
\begin{flalign*}
\varphi: \mathcal{M}^+(\mathcal{M}^-)&\longrightarrow \mathcal{Q}\\
A=\left(\begin{matrix}
a & b  \\
c & d  
\end{matrix}\right)&\longmapsto Q_A=cx^2+(d-a)xy-by^2.
\end{flalign*}
It is easy to know that $\varphi$ is injection.
For arbitrary $Q=px^2+sxy+qy^2\in \mathcal{Q}$, there is 
$M_Q=\left(\begin{matrix}
\frac{t-s}{2}  & -q  \\
p & \frac{t+s}{2}  
\end{matrix}\right)\in\mathcal{M}^+$ 
or $M_Q=\left(\begin{matrix}
-\frac{t+s}{2}  & -q  \\
p & \frac{s-t}{2}   
\end{matrix}\right)\in\mathcal{M}^-$ 
such that $\varphi(M_Q)=Q$,
where $t=\sqrt{\Delta_Q+4}$.
Thus $\varphi$ is bijective.

Let $A=\left(\begin{matrix}
a & b  \\
c & d  
\end{matrix}\right)$, 
$B=\left(\begin{matrix}
a' & b'  \\
c' & d'  
\end{matrix}\right)$, there exist 
$P=\left(\begin{matrix}
\alpha  & \beta   \\
\gamma  & \delta  
\end{matrix}\right)\in\rm{SL(2,\mathbb{Z})}$ such that $B=PAP^{-1}$. 
After a diret computation, we have that $Q_B=P^{-1}Q_A$.

Let $Q(x,y)=px^2+sxy+qy^2\in \mathcal{Q}$, 
$P=\left(\begin{matrix}
\alpha  & \beta   \\
\gamma  & \delta  
\end{matrix}\right)\in\rm{SL(2,\mathbb{Z})}$, 
$Q'(x,y)=Q(\alpha x+\beta y,\gamma x+\delta y)$.
After a direct computation, we have $M_{Q'}=P^{-1}M_QP$.

\end{proof}

Therefore, the conjugacy classes of hyperbolic elements in $\rm{SL(2,\mathbb{Z})}$ are completely
decided by the proper equivalence classes of indeﬁnite forms $\mathcal{Q}$ as follows.

$(1)$ Each proper equivalence class of indefinite forms contains a reduced form.
Section 6.7 of \cite{BV} give a method to obtain all reduced forms.

$(2)$ The set of reduced forms in the proper equivalence class of the given form $f$ is periodic.
Section 6.8 of \cite{BV} explain how to compute the cycle of $f$.

$(3)$ Let $\Delta$ be an integer with $\Delta+4$ is a perfect  square. 
We can find reduced forms that are mutually inequivalent based $(1)$ and $(2)$.
We thus obtain representatives of the conjugacy classes of hyperbolic elements in $\rm{SL(2,\mathbb{Z})}$ by Proposition \ref{conjugate}.

Let $t$ be the trace of element of $\rm{SL(2,\mathbb{Z})}$.
Next we give an example. The specific calculation process is carried out using Matlab program.

\begin{exa}
For $|t|=10$, $\Delta=t^2-4=96$.

$(1)$ The reduced forms of indefinite forms are 
\begin{flalign*}
&(1,8,-8), \ (-8,8,1),\  (-1,8,8),\  (8,8,-1),\\
&(2,8,-4), \ (-4,8,2), \ (-2,8,4), \ (4,8,-2),\\
&(3,6,-5), \ (-5,6,3), \ (-3,6,5), \ (5,6,-3),\\
&(4,4,-5), \ (-5,4,4), \ (-4,4,5), \ (5,4,-4).
\end{flalign*}

$(2)$ The proper cycles are  
\begin{flalign*}
&(1,8,-8),\ (-8,8,1); \ \ \ (-1,8,8),\ (8,8,-1).\\
&(2,8,-4),\ (-4,8,2); \ \ \ (-2,8,4),\ (4,8,-2).\\
&(3,6,-5),\ (-5,4,4), \ (4,4,-5), \ (-5.6.3); \\
&(-3,6,5),\ (5,4,-4), \ (-4,4,5), \ (5,6,-3).
\end{flalign*}

$(3)$ The corresponding conjugacy classes of $|t|=10$ in $\rm{SL(2,\mathbb{Z})}$ are represented by
\begin{flalign*}
&\left(\begin{matrix}
1 & 8  \\
1 & 9  
\end{matrix}\right),
\ 
\left(\begin{matrix}
1 & -8  \\
-1 & -9  
\end{matrix}\right),
\left(\begin{matrix}
1 & 4  \\
2 & 9  
\end{matrix}\right),
\left(\begin{matrix}
1 & -4  \\
-2 & 9  
\end{matrix}\right),
\left(\begin{matrix}
2 & 5  \\
3 & 8  
\end{matrix}\right),
\left(\begin{matrix}
2 & -5  \\
-3 & 8 
\end{matrix}\right).\\
&\left(\begin{matrix}
-9 & 8  \\
1 & -1  
\end{matrix}\right),
\ 
\left(\begin{matrix}
-9 & -8  \\
-1 & -1  
\end{matrix}\right),
\left(\begin{matrix}
-9 & 4  \\
2 & -1  
\end{matrix}\right),
\left(\begin{matrix}
-9 & -4  \\
-2 & -1 
\end{matrix}\right),
\left(\begin{matrix}
-8 & 5  \\
3 & -2 
\end{matrix}\right),
\left(\begin{matrix}
-8 & -5  \\
-3 & -2 
\end{matrix}\right).
\end{flalign*}

\end{exa}

\section{The coset representatives of centralizer groups in $\rm{SL(2,\mathbb{Z})}$}

For convenience, let $G$ be $\rm{SL(2,\mathbb{Z})}$.
In this section, we describe the coset representatives of centralizer of
$S$, $ST$, $T$ and hyperbolic elements  in $G$
by regular continued fraction.
In conclusion, most Nichols algebras over modular group are inﬁnite-dimensional except Proposition \ref{N}.

\subsection{The coset representatives of $G^S$, $G^{ST}$, $ G^{T}$ in $G$}

After a direct computation, the centralizers of one element in each conjugacy class are as follows:
\begin{flalign*}
&G^S=\langle S\rangle,\ \ G^{ST}=G^{(ST)^2}=\langle ST\rangle, \ \ G^{T^n}=\left\langle T,-I\right\rangle, \ n\neq 0.
\end{flalign*}

Firstly, the relation between the regular continued fraction and element of $G$ is as follows.
\begin{thm}\label{cf}
Let $[a_1,a_2,\cdots,a_n]$ be  regular continued fraction, and $\frac{p_j}{q_j} (1\leq j\leq n)$  
be convergent fractions. Let $p_0=1$, $q_0=0$. Then the corresponding element in $G$ is
\begin{flalign*}
N_n&=T^{a_1}ST^{-a_2}ST^{a_3}S\cdots T^{(-1)^{n-1}a_n}S\\
&=k_n\left(\begin{matrix}
p_n & (-1)^np_{n-1}  \\
q_n & (-1)^nq_{n-1} 
\end{matrix}\right),  
\ \ \  k_n=\begin{cases} 1, n\equiv 0,1 (\rm{mod}\ 4),\\ 
  -1, n\equiv 2,3 (\rm{mod}\ 4).
\end{cases}
\end{flalign*}
\end{thm}
\begin{proof}
Using mathematical induction to prove. When $n=1$,
\begin{flalign*}
N_1=T^{a_1}S=\left(\begin{matrix}
1 & a_1  \\
0 & 1 
\end{matrix}\right)
\left(\begin{matrix}
0 & -1  \\
1 & 0
\end{matrix}\right)=
\left(\begin{matrix}
a_1 & -1  \\
1 & 0
\end{matrix}\right)=
\left(\begin{matrix}
p_1 & -p_{0}  \\
q_1 & q_{0}
\end{matrix}\right), 
\end{flalign*}

Assume 
\begin{flalign*}
N_{n-1}=k_{n-1}\left(\begin{matrix}
p_{n-1} & (-1)^{n-1}p_{n-2}  \\
q_{n-1} & (-1)^{n-1}q_{n-2} 
\end{matrix}\right).
\end{flalign*}
Then
\begin{flalign*}
N_n=&N_{n-1}T^{(-1)^{n-1}a_n}S\\
=&k_{n-1}\left(\begin{matrix}
p_{n-1} & (-1)^{n-1}p_{n-2}  \\
q_{n-1} & (-1)^{n-1}q_{n-2} 
\end{matrix}\right)
\left(\begin{matrix}
(-1)^{n-1}a_n & -1  \\
1 & 0
\end{matrix}\right).\\
=&k_{n-1}\left(\begin{matrix}
(-1)^{n-1}(a_np_{n-1}+p_{n-2}) & -p_{n-1}  \\
(-1)^{n-1}(a_nq_{n-1}+q_{n-2} ) & -q_{n-1} 
\end{matrix}\right)\\
=&k_n\left(\begin{matrix}
p_n & (-1)^np_{n-1}  \\
q_n & (-1)^nq_{n-1} 
\end{matrix}\right).
\end{flalign*}

\end{proof}

Let $i$ be a primitive $2$-th root of unity, and $\omega=e^{\frac{\pi }{3}i}$ be a primitive $6$-th root of unity.
Considering $G$ act on $C^*$ as modular group, we have that 
\begin{pro}\label{S}
$\langle S\rangle$ is the stable subgroup of $i$. $\langle ST\rangle$ is the stable subgroup of $\omega $.
$G^T$ is the stable subgroup of $\infty$.
\end{pro}

Next we describe all cosets of $G$ over $G^S$, $G^{ST}$, $ G^{T}$  by starting from regular continued fraction.

\begin{thm}\label{rep}
$(1)$ A  representative of complete cosets of $G$ over $G^S$ is
\begin{flalign}\label{complete cosets}
T^{a_1}ST^{-a_2}S\cdots T^{(-1)^{n-1}a_n}ST^{a_{n+1}},\ a_{n+1}\in \mathbb{Z},
\end{flalign}
where $[a_1,a_2,\cdots,a_n]$ is the regular continued fraction of arbitrary rational number, and 
remove the case $(\ref{coset})$.

$(2)$ A  representative of complete cosets of $G$ over  $G^{ST}$ is $T^m$ $(m\in \mathbb{Z})$ and $(\ref{complete cosets})$,
where $[a_1,a_2,\cdots,a_n]$ is the regular continued fraction of arbitrary fraction, and 
remove the case  $(\ref{coset ST})$.

$(3)$ A  representative of complete cosets of $G$ over $G^{T}$ is
\begin{flalign*}
I, T^{a_1}ST^{-a_2}S\cdots T^{(-1)^{n-1}a_n}S,
\end{flalign*}
where $[a_1,a_2,\cdots,a_n]$ is the regular continued fraction of arbitrary rational number.
\end{thm}
\begin{proof}
$(1)$ By Proposition \ref{S}, we have a one-to-one map
\begin{flalign*}
G/G^S&\longrightarrow Gi\\
gG^S&\longmapsto gi
\end{flalign*}
Therefore, fingding different left coset  $gG^S$ are fingding different $gi$.

Let $g=\left(\begin{matrix}
a & b  \\
c & d 
\end{matrix}\right)\in G$.
We have gcd$(a,c)=$gcd$(c,d)=1$.
Because
\begin{flalign*}
&gi=\frac{ai+b}{ci+d}=\frac{ac+bd}{c^2+d^2}+\frac{1}{c^2+d^2}i, 
\end{flalign*}
we know the fact that 
\begin{flalign}\label{coset}
&\left(\begin{matrix}
a & b \\
c & d 
\end{matrix}\right)i
=\left(\begin{matrix}
b & -a \\
d & -c 
\end{matrix}\right)i.
\end{flalign}
That is,
$\left(\begin{matrix}
a & b \\
c & d 
\end{matrix}\right)$ and 
$\left(\begin{matrix}
b & -a \\
d & -c 
\end{matrix}\right)$ belong to the same coset.

For $(a,c)$, the solution of equation $ad-bc=1$ is not unique.
Therefore, $gT^tG^S(t\in \mathbb{Z})$ corresponding different points in the orbit $Gi$.

Next we dissus  $\frac{a}{c}$ case by case.

$(i)$ $c=0$, The corresponding $T^nS$ belong to case $(ii)$. 

$(ii)$ $c\neq 0$.
 Let $\frac{a}{c}=[a_1,a_2,\cdots,a_n]$.
By Theorem \ref{rep}, there is
\begin{flalign*}
T^{a_1}ST^{-a_2}S\cdots T^{(-1)^{n-1}a_n}S.
\end{flalign*}
Whence $\frac{a}{c}$ corresponding to $(\ref{complete cosets})$.
But in the cosets of $G$ over $G^S$, we need remove the case $(\ref{coset})$.

$(2)$ Similarly, fingding different left coset $gG^{ST}$ are fingding different $g\omega$.
We know that
\begin{align}\label{coset ST}
  \begin{split}
gG^{ST}=&\pm\{g,gST,g(ST)^2\}.
\end{split}
\end{align}

Next we dissus  $\frac{a}{c}$ case by case.

$(i)$  $c=0$.
$\{T^n|n\in \mathbb{Z} \}$ are the representative of cosets of $G$ over $G^{ST}$.

$(ii)$  $c\neq 0$.

If $\frac{a}{c}$ is integer. Let $\frac{a}{c}=[m]$.
The corresponding elements of $G$ are $T^mST^{n}$.
When $n=0,1$, $T^mST^{n}$ belong to case $(i)$.
When $n\neq0,1$, $T^mST^{n}$ belong to the next case by equation \ref{coset ST}.

If $\frac{a}{c}$ is fraction.  Let $\frac{a}{c}=[a_1,a_2,\cdots,a_n]$.
Then $\frac{a}{c}$ corresponding to $(\ref{complete cosets})$.
But in the cosets of $G$ over $G^S$, we need remove the case $(\ref{coset ST})$.

$(3)$ Fingding different left cosets of $gG^T$ are fingding different $g \infty$.

$(i)$ $g\infty=\infty$, take $I$ as the representative element of coset.

$(ii)$ $g\infty=\frac{c}{a}\neq \infty$.
Let $\frac{a}{c}=[a_1,a_2,\cdots,a_n]$. The corresponding representative elements of coset is 
\begin{flalign*}
T^{a_1}ST^{-a_2}S\cdots T^{(-1)^{n-1}a_n}S.
\end{flalign*}
\end{proof}

\subsection{The coset representatives of centralizer of
hyperbolic elements  in $G$}

\begin{lem}\label{stable}
Let $M$ be hyperbolic, $\alpha$ be the fixed point of $M$. Then $G^M$ is the stable subgroup of $\alpha$.
\end{lem}
\begin{proof}
Let $G_\alpha=\{g\in G\ |\ g\cdot \alpha=\alpha\}$.
Let $M=\left(\begin{matrix}
a & b  \\
c & d 
\end{matrix}\right)$, $M\cdot \alpha=\alpha$. Then
$\left(\begin{matrix}
\alpha  \\
1
\end{matrix}\right)$ is the eigenvector of $M$ belonging to eigenvalue $c\alpha+d$.
$\forall X\in G^M$, we have that $X\left(\begin{matrix}
\alpha  \\
1
\end{matrix}\right)$ is the eigenvector of $M$ belonging to eigenvalue $c\alpha+d$.
Whence, there exist $0\neq \mu \in \mathbbm{k}$ such that 
\begin{flalign*}
X\left(\begin{matrix}
\alpha  \\
1
\end{matrix}\right)=\mu\left(\begin{matrix}
\alpha  \\
1
\end{matrix}\right).
\end{flalign*}
That is, $X\cdot \alpha=\alpha$. Thus $G^M\subseteq G_\alpha$.

Let $g=\left(\begin{matrix}
p & q  \\
r & s
\end{matrix}\right)\in G_\alpha$. By $g\cdot \alpha=\alpha$, $M\cdot \alpha=\alpha$, we have that
\begin{flalign*}
r\alpha^2+(s-p)\alpha-q=0, \ \ c\alpha^2+(d-a)\alpha-b=0.
\end{flalign*}
Because $\alpha$ satisfies the unique quadratic equation with integer coefficients,
there exist $\lambda \in \mathbbm{k}$ such that 
\begin{flalign*}
r=\lambda c,\ s-p=\lambda(d-a),\ q=\lambda b.
\end{flalign*}
Therefore, $gM=Mg$. That is $G_\alpha\subseteq G_M$.

\end{proof}

Here are a few lemmas about regular continued fractions, and the proof is fairly straightforward.

\begin{lem}\label{countdown}
Let $x=[c_0,c_1,c_2,\cdots]$ be regular continued fractions.  Then 

$(1)$ When $x>0$, $c_0\geq 0$. 
$$
\frac{1}{x}=\left\{
  \begin{aligned}
        &[0,c_0,c_1,c_2,\cdots], &c_0>0,\\
        &[c_1,c_2,c_3,\cdots], & c_0=0,
  \end{aligned}
\right.
$$

$(2)$ 
$$
-x=\left\{
 \begin{aligned} 
&[-c_0-1,1,c_1-1,c_2,c_3,\cdots],  &  c_1>1, \\
&  [-c_0-1,c_2+1,c_3,c_4,\cdots], &  c_1=1,
\end{aligned}
\right.
$$

\end{lem}

\begin{lem}\label{ccf}
Let $[a_1,a_2,\cdots,a_n]$ be a regular continued fraction, and $\frac{p_j}{q_j} (1\leq j\leq n)$  
be convergent fractions. Let $\gamma$ be positive integer. Then
\begin{flalign*}
[a_1,a_2,\cdots,a_k,\gamma]=\frac{p_k\gamma+p_{k-1}}{q_k\gamma+q_{k-1}},\ \ k=1,2,\cdots,n.
\end{flalign*}
where $p_0=1$, $q_0=0$.
\end{lem}

\begin{lem}\label{G orbit}
Let $\alpha=[a_1,a_2,\cdots,a_m,\overline{b_1,b_2,\cdots,b_n}]$ be periodic regular continued fraction. Then
$$
G\cdot \alpha=[c_1,c_2,\cdots,c_k,\overline{b_1,b_2,\cdots,b_n}], k\geq 1,\left\{
  \begin{aligned}
        &\ k\equiv m (\rm{mod}\ 2),   n\ is\  even, \\
        &\ k\in \mathbb{Z} ,   n\ is\  odd,
  \end{aligned}
\right.
$$
where $c_1$ is integer, $c_2,c_3,\cdots,c_k$ are positive integers. 
\end{lem}

\begin{proof}
Let $\mathcal{O}$ be the set of right side of equation. 
$G\cdot \alpha\subseteq \mathcal{O}$ can be concluded form Lemma $\ref{countdown}$.

Let $\beta=[c_1,c_2,\cdots,c_k,\overline{b_1,b_2,\cdots,b_n}]\in \mathcal{O}$, $\gamma=[\overline{b_1,b_2,\cdots,b_n}]$, and 
\begin{align}\label{abc}
  \begin{split}
&[c_1,c_2,\cdots,c_i]=\frac{p_i}{q_i},\ i\in \mathbb{I}_{1,k},
\  [a_1,a_2,\cdots,a_j]=\frac{p_j'}{q_j'},\ j\in \mathbb{I}_{1,m},\\
&[b_1,b_2,\cdots,b_l]=\frac{p_l''}{q_l''},\ l\in \mathbb{I}_{1,n}.
  \end{split}
\end{align}
Then by Lemma $\ref{ccf}$, we conclude that 
\begin{flalign*}
\beta=\left(\begin{matrix}
p_k & p_{k-1}  \\
q_k & q_{k-1}
\end{matrix}\right)\gamma, \ \
\alpha=\left(\begin{matrix}
p_m' & p_{m-1}'  \\
q_m' & q_{m-1}'
\end{matrix}\right)\gamma,\ \ 
\gamma=\left(\begin{matrix}
p_n'' & p_{n-1}''  \\
q_n' & q_{n-1}''
\end{matrix}\right)\gamma.
\end{flalign*}
When  $k\equiv m (\rm{mod}\ 2)$, let
\begin{flalign*}
g=(-1)^m\left(\begin{matrix}
p_k & p_{k-1}  \\
q_k & q_{k-1}
\end{matrix}\right)
\left(\begin{matrix}
q_{m-1}'  & -p_{m-1}'  \\
-q_m' & p_m'
\end{matrix}\right).
\end{flalign*}
When  $n$  is odd, $k\not\equiv m (\rm{mod}\ 2)$, let
\begin{flalign*}
g=(-1)^m\left(\begin{matrix}
p_k & p_{k-1}  \\
q_k & q_{k-1}
\end{matrix}\right)
\left(\begin{matrix}
p_n'' & p_{n-1}''  \\
q_n' & q_{n-1}''
\end{matrix}\right)
\left(\begin{matrix}
q_{m-1}'  & -p_{m-1}'  \\
-q_m' & p_m'
\end{matrix}\right).
\end{flalign*}
We have that $g\in G$ and $\beta=g\cdot \alpha$. Thus $\mathcal{O}\subseteq G\cdot \alpha$.

\end{proof}

\begin{thm}\label{hyp}
Let $M$ be hyperbolic, $\alpha=[a_1,a_2,\cdots,a_m,\overline{b_1,b_2,\cdots,b_n}]$ be the fixed point of $M$. 
Then a representative of complete cosets of $G$ over $G^M$ is
\begin{flalign*}
&\left\{(-1)^m\left(\begin{matrix}
p_k & p_{k-1}  \\
q_k & q_{k-1}
\end{matrix}\right)
\left(\begin{matrix}
q_{m-1}'  & -p_{m-1}'  \\
-q_m' & p_m'
\end{matrix}\right)\middle| k\equiv m (\rm{mod}\ 2)\right\},\\
&\left\{ (-1)^m\left(\begin{matrix}
p_k & p_{k-1}  \\
q_k & q_{k-1}
\end{matrix}\right)
\left(\begin{matrix}
p_n'' & p_{n-1}''  \\
q_n' & q_{n-1}''
\end{matrix}\right)
\left(\begin{matrix}
q_{m-1}'  & -p_{m-1}'  \\
-q_m' & p_m'
\end{matrix}\right)\middle|n\ is\ odd,\  k\not\equiv m (\rm{mod}\ 2)\right\},      
\end{flalign*}
where $[c_1,c_2,\cdots,c_k]$ is the regular continued fraction of arbitrary rational number, 
$k\geq 1,\ k\in \mathbb{Z}$,
and satisfying equation $(\ref{abc})$.
\end{thm}
\begin{proof}
The claim follows by Lemmas $\ref{stable}$ and $\ref{G orbit}$.
\end{proof}

\begin{cor}
Based on the classification of representative of complete cosets of $G$ in Theroem \ref{rep} and \ref{hyp},
we conclude that the Nichols algebra $\mathcal{B}(\mathcal{O}_g,V)$ is infinite-dimensional, where $\pm I\neq g\in G$,
 $V$ is a simple $G^g$ module.
\end{cor}

\begin{pro}\label{N}
When $g=\pm I$, $(\rho,V)$ is a simple $G$ module. Then

$dim \mathcal{B}(\mathcal{O}_g,V)=\left\{
  \begin{aligned}
        &\infty,   \ g=I, \ or\ g=-I, \  \rho(-I)=I_V,  \\
        &2^{dim V},  \  g=-I, \  \rho(-I)=-I_V.
  \end{aligned}
\right.$

where $I_V$ is the identity linear transformation on $V$.
\end{pro}
\begin{proof}
When $g=\pm I$, the simple Yetter-Drinfeld module in $^G_G\mathcal{YD}(\mathcal{O}_g)$ is $M(g,V)$.
Let $v\in M(g,V)$. The comodule structure is $\delta(v)=g\otimes v$. Thus the braiding is 

$c(v\otimes v)=\left\{
  \begin{aligned}
        &v\otimes v,   \ g=I, \ or\ g=-I, \  \rho(-I)=I_V,  \\
        &-v\otimes v,  \  g=-I, \  \rho(-I)=-I_V.
  \end{aligned}
\right.$

Whence we get the claim.
\end{proof}

\begin{center}
$\mathbf{ACKNOWLEDGMENT}$
\end{center}

The author would  like to thank Professor Gongxiang Liu for useful comments and suggestions. 
The  author is supported by the  NSFC (Grant No. 12201164).

\bibliographystyle{amsalpha}

\end{document}